\documentclass[12pt]{article}
\usepackage{amsfonts}
\usepackage{amssymb}
\usepackage{amsmath}
\usepackage{amsthm}
\usepackage{euscript}
\usepackage{graphics}
\usepackage{verbatim}
\usepackage{color}

\usepackage[margin=1.1in]{geometry}

\newcommand{\Y}{\mathbb Y}

\newcommand{\eps}{\varepsilon}

\newtheorem{theorem}{Theorem}[section]
\newtheorem{proposition}[theorem]{Proposition}
\newtheorem{lemma}[theorem]{Lemma}
\newtheorem{definition}[theorem]{Definition}
\newtheorem{conjecture}[theorem]{Conjecture}
\newtheorem{corollary}[theorem]{Corollary}

\theoremstyle{remark}
\newtheorem{remark}[theorem]{Remark}

\renewcommand{\S}{\mathfrak S}

\newcommand{\pr}{\nearrow}

\renewcommand{\d}{\mathbf d_{\rm var}}
\newcommand{\di}{\mathbf d_{\infty}}

   \title{Stochastic monotonicity in Young graph and Thoma theorem.}
   \author{Alexey Bufetov\thanks{International Laboratory of Representation Theory and Mathematical Physics, Department of Mathematics, Higher School of Economics, Moscow, Russia, and Institute for
Information Transmission Problems of Russian Academy of Sciences, Moscow, Russia. E-mail:
alexey.bufetov@gmail.com} \and Vadim Gorin\thanks{Department of Mathematics, Massachusetts
Institute of Technology, Cambridge, MA, USA, and
 Institute for Information Transmission Problems of Russian Academy of Sciences, Moscow, Russia. E-mail: vadicgor@gmail.com}
}

\begin{document}
\maketitle

\begin{abstract}
 We show that the order on probability measures, inherited from
 the dominance order on the Young diagrams, is preserved under
 natural maps reducing the number of boxes in a diagram by $1$. As a
 corollary we give a new proof of the Thoma theorem on the
 structure of characters of the infinite symmetric group.

We present several conjectures generalizing our result. One of them (if it is true) would imply the
Kerov's conjecture on the classification of all homomorphisms from the algebra of symmetric
functions into $\mathbb R$ which are non-negative on Hall--Littlewood polynomials.
\end{abstract}

\section{Introduction}

\subsection{Problem setup and results}

For a number $n=0,1,2,\dots$, a partition $\lambda$ of $n$ is a sequence of integers $\lambda_1\ge\lambda_2\ge\dots\ge
0$ such that $|\lambda|=n$, where $|\lambda|=\sum_{i=1}^\infty\lambda_i$. We identify a partition $\lambda$ with the
Young diagram, which is a collection of $|\lambda|$ boxes with positive coordinates $(i,j)$ forming the following
set
$$
 \{(i,j)\subset \mathbb{Z}_{>0}\times\mathbb{Z}_{>0} \mid j\le \lambda_i \}.
$$
When drawing pictures we adopt the notation that the first index $i$ increases as we move down,
while the second index $j$ increases as we move to the right, cf.\ Figures \ref{Figure_cover} and
\ref{Figure_4diag}.

The Young graph $\Y=\bigcup_{n=0}^{\infty}\Y_n$ is a graded graph such that the vertices of $\Y_n$
are all partitions of $n$. In particular, $\Y_0$ contains only the empty partition
$\emptyset=(0,0,\dots)$. An edge joins $\lambda\in \Y_{n}$ with $\mu\in \Y_{n-1}$, $n\ge 1$, if and
only if $\lambda$ differs from $\mu$ by the addition of a single box, which we denote
$\mu\nearrow\lambda$.

For a Young diagram $\lambda$, its \emph{dimension}\footnote{The name originates in the fact that
$\dim(\lambda)$ coincides with the dimension of the irreducible representation of the symmetric
group $\S_n$ indexed by $\lambda$. Here $n=|\lambda|$.} denoted by $\dim(\lambda)$ is the number of
oriented paths in $\Y$ which start at $\emptyset$ and end at $\lambda$.

Let $M_n$ be a probability measure on $\Y_n$. Its \emph{projection} onto $\Y_{n-1}$ denoted by
$\pi^n_{n-1}M_n$ is defined via
$$
 (\pi^n_{n-1}M_n)(\mu)=\sum_{\lambda\in\Y_{n}:\, \mu\pr\lambda}
 \frac{\dim(\mu)}{\dim(\lambda)} M_n(\lambda).
$$
The definition readily implies that $\pi^n_{n-1}M_n$ is a probability measure. Iterating the maps
$M_n\mapsto \pi^n_{n-1}M_n$ one similarly defines the projection of $M_n$ onto $M_k$, $0\le k<n$,
denoted by $\pi^n_k M_n$.

\begin{definition} A sequence of measures $\{M_n\}_{n=0}^\infty$ is
called a \emph{coherent system} on $\Y$ if each $M_n$, $n=0,1,\dots$
is a probability measure on $\Y_n$ and for any $0\le k <n$ the
measure $M_k$ is the projection of $M_n$ onto $\Y_k$, i.e.\
$M_k=\pi^n_k M_n$.
\end{definition}

In last 40 years coherent systems on $\Y$ were enjoying lots of interest due to their connections
to several seemingly unrelated topics. First, one can show that they are in bijection with
normalized characters for the infinite symmetric group and have a close relation to the finite
factor and spherical representations of the latter, see \cite{VK_S}, \cite{Kerov_book},
\cite{Ok_S}. Second, there is a correspondence between such systems of measures and totally
positive upper triangular Toeplitz matrices, see \cite{Thoma}, \cite[Section 2.2]{Kerov_book},
\cite{Ok_S}. Third, they are naturally linked to combinatorial objects appearing in the study of
the Robinson--Schensted--Knuth correspondence, cf.\ \cite{VK_RSK}. Finally, several instances of
these systems, e.g.\ the celebrated Plancherel distributions, exhibit a remarkable asymptotic
behavior as $n\to\infty$ and, in particular, numerous connections to random matrices, see
\cite{BDJ}, \cite{BOO}, \cite{Ok}, \cite{J-Annals}, \cite{Kerov_CLT}, \cite{IvanovOlshanski}.

The classification of all coherent systems on $\Y$ (in an equivalent form) is now known as
\emph{Thoma theorem}. Its formulation uses the symmetric functions notations which we now
introduce. Let $\Lambda$ be the algebra of all symmetric functions in countably many variables
$x_1,x_2,\dots$, see e.g.\ \cite[Chapter 1, Section 2]{M}. One way to define $\Lambda$ is as an
algebra (over $\mathbb R$) of polynomials in Newton power sums $p_k$, $k=1,2,\dots$
$$
 p_k=x_1^k+x_2^k+x_3^k+\dots.
$$
An important linear basis of $\Lambda$ is formed by \emph{Schur symmetric functions} $s_\lambda$,
$\lambda\in\Y$, and we refer to \cite[Chapter 1, Section 3]{M} for the exact definition and
properties of $s_\lambda$.

We also define $\Omega$ to be the set of all pairs of sequences
$(\alpha,\beta)=(\alpha_1\ge\alpha_2\ge\dots\ge 0,\beta_1\ge\beta_2\ge\dots\ge 0),$ such that
$\sum_{i=1}^\infty (\alpha_i+\beta_i)\le 1$.

\begin{theorem}[Thoma theorem, cf.\ \cite{Thoma}, \cite{VK_S}, \cite{Ok_S}, \cite{KOO}, \cite{V_new}] \label{Theorem_Thoma} The set of all coherent
systems is a (Choquet) simplex, whose extreme points are parameterized by elements of $\Omega$. The
extreme system of measures $\{M_n^{(\alpha,\beta)}\}_{n=0}^{\infty}$ parameterized by
$(\alpha,\beta)\in\Omega$ is given by
\begin{equation}
\label{eq_ext_measure_def}
 M_n^{(\alpha,\beta)}(\lambda)=\dim(\lambda)
 s_\lambda(\alpha,\beta),
\end{equation}
where $s_\lambda(\alpha,\beta)$ is the image of $s_\lambda$ under
the algebra homomorphism from $\Lambda$ to $\mathbb R$ defined on
power sums $p_k$ via
\begin{equation}
\label{eq_specialization}
 p_1\mapsto p_1(\alpha,\beta)=1,\quad p_k\mapsto p_k(\alpha,\beta)= \sum_{i=1}^\infty \alpha_i^k + (-1)^{k-1}
 \sum_{i=1}^\infty \beta_i^k,\, k=1,2,\dots.
\end{equation}
\end{theorem}

One of the aims of our article is to give a new proof of Theorem \ref{Theorem_Thoma} based on a
monotonicity--preservation property that we will now present. Our proof of Thoma theorem is based
on the combinatorial and probabilistic ideas only; other existing proofs use highly nontrivial
analytic \cite{Thoma} or algebraic \cite{VK_S,KOO,Ok_S} methods (see, however, \cite{V_new}).
 We hope that the strategy used in our
proof of Theorem \ref{Theorem_Thoma} could be used in the future to establish the validity of a
generalization of Theorem \ref{Theorem_Thoma} known as the \emph{Kerov's conjecture}, see Section
\ref{Section_generalization_intro} for more details.

\smallskip

Let us equip $\Y_n$ with a partial order known as \emph{dominance order}. For $\lambda,\mu\in\Y_n$
we write $\lambda\ge \mu$, if for all $k=1,2,\dots$ we have
$$
\lambda_1+\lambda_2+\dots+\lambda_k\ge \mu_1+\mu_2+\dots+\mu_k.
$$
Further, we say that a measure $\rho$ on $\Y_n$ is an atom if its
support consists of a single element and write $\sup(\rho)$ for this
element. Note that we allow the mass of $\rho$ to be different from
$1$ here.

\begin{definition} \label{Def_dominance} Let $\rho$ and $\rho'$ be two
measures on $\Y_n$ of the same total mass, i.e. $\rho(\Y_n)=\rho'(\Y_n)$. We say that $\rho$
stochastically dominates $\rho'$ and write $\rho\ge \rho'$, if there exist $k>0$ and $2k$ measures
$\rho_1,\dots,\rho_k$, $\rho'_1,\dots,\rho'_k$, such that $ \rho=\sum_{i=1}^k \rho_i$,
$\rho'_i=\sum_{i=1}^k \rho'_i,$ and, moreover, $\rho_i$, $\rho'_i$ are atoms of the same mass and
with $\sup(\rho_i)\ge \sup(\rho'_i)$ for each $i=1,\dots,k$.
\end{definition}
Informally, Definition \ref{Def_dominance} means that $\rho$ can be
obtained from $\rho'$ by moving masses up with respect to our
partial order.

\begin{theorem} \label{Theorem_monotonicity}
  Take $0\le k<n$ and let $\rho$ and $\rho'$ be two
measures on $\Y_n$ of the same total mass. If $\rho\ge \rho'$, then the same is true for their
projections on $\Y_k$, i.e.\ $\pi^n_k \rho\ge \pi^n_k \rho'$.
\end{theorem}
 We prove Theorem \ref{Theorem_monotonicity} in Section \ref{Section_monoton}. Our proof is based
 on inequalities for the dimensions in Young graph presented in Corollary
 \ref{Cor_elementary_dimensions}. We also explain that these inequalities admit natural
 generalizations to the statements about the monomial positivity of certain quadratic expressions
 in Schur polynomials; we do not know a proof for the latter monomial positivity and present it as
 Conjecture \ref{Conj_monomial_positive}.

 In
 Section \ref{Section_Thoma}, we combine Theorem \ref{Theorem_monotonicity}
 with the Law of Large Numbers for a subclass
 of extreme coherent systems and deduce Theorem \ref{Theorem_Thoma}.
 Finally, in Section \ref{Section_LLN} we recall the aforementioned Law of Large Numbers and explain
 several known strategies of its proof.

\subsection{$t$--Deformation and Kerov's conjecture}
\label{Section_generalization_intro}

Theorem \ref{Theorem_Thoma} is known (see e.g.\ \cite{Kerov_book}) to be equivalent to the
following description of all Schur--positive homomorphisms from $\Lambda$ into $\mathbb R$.
\begin{theorem} \label{Theorem_homo}
 The set of algebra homomorphisms $\varrho:\Lambda\to\mathbb R$ normalized by the condition
 $\varrho(p_1)=1$ and such that $\varrho(s_\lambda)\ge 0$ for all $\lambda\in\Y$, is in bijection with
 $\Omega$.
 The homomorphism corresponding to $(\alpha,\beta)\in\Omega$ is defined by its values
 on power sums $p_k$
  \begin{equation}
\label{eq_specializations2}
 p_1\mapsto p_1(\alpha,\beta)=1,\quad p_k\mapsto p_k(\alpha,\beta)= \sum_{i=1}^\infty \alpha_i^k + (-1)^{k-1}
 \sum_{i=1}^\infty \beta_i^k,\, k=1,2,\dots.
\end{equation}
\end{theorem}

A natural way to generalize Theorem \ref{Theorem_homo} is by replacing Schur functions $s_\lambda$
by other classes of symmetric functions. Kerov conjectured 20 years ago that when $s_\lambda$ are
replaced by their celebrated $(q,t)$-deformation --- Macdonald polynomials $M_\lambda(\cdot;q,t)$
--- then (for $0\le q <1$, $0\le t<1$) the Macdonald--positive homomorphisms are still in bijection
with elements of $\Omega$. The conjectural correspondence is established through the formulas very
similar to \eqref{eq_specializations2}, see \cite[Chapter II, Section 9]{Kerov_book} for the
details. The completeness of the Kerov's list of homomorphisms is still an open problem (though it
is relatively easy to show that all these homomorphisms are indeed Macdonald--positive, see e.g.\
\cite[Section 2.2.1]{BC}). Recently, these homomorphisms have been actively used for the asymptotic
analysis of a variety of probabilistic systems in the framework of Macdonald processes, see
\cite{BC}, \cite{BCGS}.

The $q=0$ versions of Macdonald polynomials are the Hall--Littlewood polynomials, see \cite{M}.
This particular case of the Kerov's conjecture is especially interesting, since when $t=p^{-1}$ the
conjecture is equivalent to the (conjectural) classification of all conjugation invariant ergodic
measures on infinite uni--uppertriangular matrices over a finite field with $p$ elements $\mathbb
F_p$, see \cite[Section 4]{GKV}.

Recently a progress on the $t$--deformation of Theorem \ref{Theorem_Thoma} (equivalent to the
Hall-Littlewood case of Kerov's conjecture, see \cite[Section 4]{GKV} and \cite[Section
4.2]{Fulman} for the details) was achieved in \cite{BufPet}, where the Law of Large Numbers for the
measures arising in it was proved. We thus hope that our approach to the proof of Theorem
\ref{Theorem_Thoma} can be extended to the Hall--Littlewood case of Kerov's conjecture. More
precisely, if one tries to mimic our approach, then the conjecture at $t=p^{-1}$ reduces to the
following inequality.

Let $U_n$ be the group of all uni--uppertriangular matrices over $\mathbb F_p$. Note that for each
$u\in\ U_n$ all its eigenvalues are $1$s and thus we can assign to it a unique Young diagram
$\mathcal J (u)\in\Y_n$ whose row lengths are sizes of the blocks in Jordan Normal Form of $u$. We
define
$$
 \dim_t(\lambda)=|\{ u\in U_n\mid \mathcal J(u)=\lambda\}|.
$$
Further, for any $u\in U_{n}$ we set $u^{(n-1)}\in U_{n-1}$ to be its top--left $(n-1)\times(n-1)$
corner, and define for $\mu\in\Y_{n-1}$, $\lambda\in\Y_n$
$$
 \dim_t(\mu\nearrow\lambda)=|\{u\in U_n \mid \mathcal J(u^{(n-1)})=\mu,\, \mathcal J(u)=\lambda \}|.
$$
We remark that \cite[Theorem 2.3]{Borodin_Haar} (see also \cite{Kir}) gives an explicit formula for
the ratio $\frac{\dim_t(\mu\nearrow\lambda)}{\dim_t(\mu)}$, which, in particular, implies that
$\dim_t(\mu\nearrow\lambda)$ vanishes unless $\mu\nearrow\lambda$.

\begin{conjecture}
 \label{Conj_elementary_dimensions_t}
 Let $\lambda,\hat \lambda\in Y_n$ and $\mu,\hat \mu\in Y_{n-1}$ be two
 pairs of Young diagrams, such that both $\lambda$,$\hat \lambda$  and $\mu$,$\hat \mu$ differ
 by the move of box $(i,j)$ into the position $(\hat i,\hat j)$ with $\hat i> i$.
 Further, assume that $\lambda\setminus\mu=\hat \lambda\setminus\hat \mu=(r,c)$, cf.\ Figures
 \ref{Figure_cover} and
 \ref{Figure_4diag}.
 If $r<i$ then
 \begin{equation}
 \label{eq_monoton_elementary_1_dim_t}
  \frac{\dim_t(\hat\mu\nearrow\hat\lambda)}{\dim_t(\hat \lambda)} \ge
  \frac{\dim_t(\mu\nearrow\lambda)}{\dim_t(\lambda)}.
 \end{equation}
 If $r>\hat i$, then
 \begin{equation}
 \label{eq_monoton_elementary_2_dim_t}
  \frac{\dim_t(\hat\mu\nearrow\hat\lambda)}{\dim_t(\hat \lambda)} \le \frac{\dim_t(\mu\nearrow\lambda)}{\dim_t(\lambda)}.
 \end{equation}
\end{conjecture}
This conjecture can be also restated as a certain inequality for the values of Hall--Littlewood
polynomials, and its generalization is formulated below in Conjecture
\ref{Conjecture_elementary_move}. Computer checks supply the validity of these conjectures, but we
have not found a proof.

\smallskip

At $t=1$ the Hall--Littlewood polynomials turn into the monomial symmetric functions and this case
of the Kerov's conjecture is equivalent to the Kingman's classification theorem for exchangeable
partition structures on $\mathbb Z_{>0}$, see \cite[Chapter I]{Kerov_book}. Both ingredients of our
approach, which are the $t=1$ versions of Conjecture \ref{Conj_elementary_dimensions_t} and the Law
of Large Numbers for the extreme coherent systems are especially simple and transparent in this
case. Thus, by mimicking our proof of Theorem \ref{Theorem_Thoma} one can also get a new proof of
the Kingman's classification theorem \cite{Ki}.

\bigskip

\noindent{\bf Acknowledgements.} A.~B.\ was partially supported by Simons Foundation-IUM
scholarship, by ``Dynasty'' foundation, and by the RFBR grant 13-01-12449.  V.~G.\ was partially
supported by the NSF grant DMS-1407562.

\section{Monotonicity in Young graph}
\label{Section_monoton}

This section is devoted to the proof of Theorem \ref{Theorem_Thoma}.

\subsection{Elementary moves}
\label{Section_elementary}
 First, let us introduce several
additional notations. We say that two distinct Young diagrams $\lambda\in \Y_n$ and $\hat
\lambda\in \Y_n$ differ by the move of box $(i,j)$ into the position $(\hat i,\hat j)$, if there
exists $\mu\in\Y_{n-1}$ such that $\mu=\lambda\setminus (i,j)=\hat \lambda\setminus(\hat i,\hat
j)$, see Figure \ref{Figure_cover} for an illustration. Note that we should have $\hat i\ne i$
Further, if $\hat i>i$, then $\lambda\ge\hat \lambda$ and if $\hat i<i$, then $\lambda\le\hat
\lambda$.

\begin{figure}[t]
\begin{center}
 {\scalebox{1.0}{\includegraphics{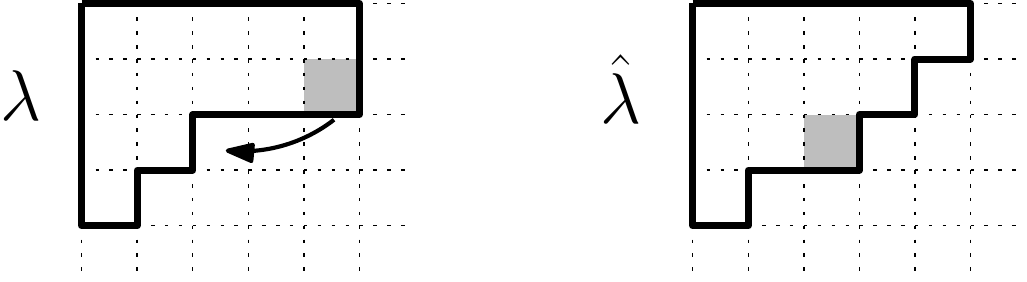}}}
 \caption{Young diagrams $\lambda$ and $\hat \lambda$ differing by the move of the box $(2,5)$ into
 the position $(3,3)$. Here $\lambda\ge\hat \lambda$ and also $\lambda\succ\hat \lambda$.
 \label{Figure_cover}}
\end{center}
\end{figure}

 Recall that for a Young diagram
$\lambda$, the numbers $\lambda'_1\ge\lambda'_2\ge\dots$ are defined as the column lengths of
$\lambda$, formally
$$
 \lambda'_j=|\{i\in\mathbb Z_{>0}: \lambda_i\ge j \}|.
$$
We also set $\ell(\lambda)$ to be the
number of non-zero rows in $\lambda$, i.e.\ $\ell(\lambda)=\lambda'_1$.

We evoke the ($N$--variable version of) Schur
symmetric function $s_\lambda$.  For any $N=1,2,\dots$ and Young diagram
$\lambda\in\Y$ such that $\ell(\lambda)\le N$, we have
$$
 s_\lambda(x_1,\dots,x_N)= \frac{\det_{i,j=1}^N\left[
 x_i^{\lambda_j+N-j}\right]}{\prod_{1\le i<j\le N} (x_i-x_j)}.
$$
Finally, we use the notation $1^N$ for
$(\underbrace{1,\dots,1}_{N})$.

Our proof of Theorem \ref{Theorem_monotonicity} relies on the
following statement.

\begin{figure}[t]
\begin{center}
 {\scalebox{1.0}{\includegraphics{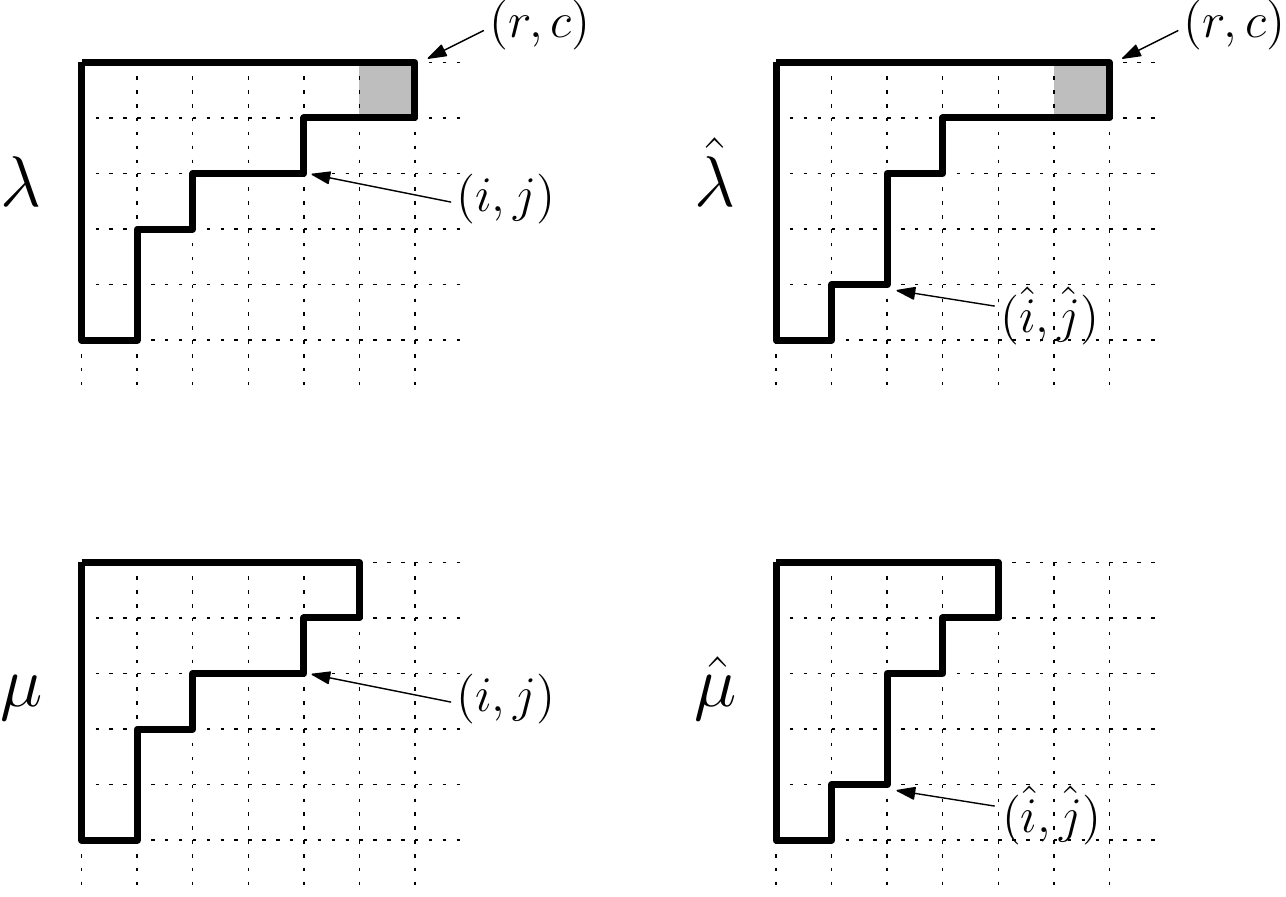}}}
 \caption{An example of Young diagrams $\lambda,\hat \lambda$ and $\mu,\hat \mu$ as in Proposition
 \ref{Prop_elementary_move}. Here the gray box is $(r,c)=(1,6)$, and $(i,j)=(2,4)$,
 $(\hat i,\hat j)=(4,2)$.
 \label{Figure_4diag}}
\end{center}
\end{figure}

\begin{proposition} \label{Prop_elementary_move}
 Let $\lambda,\hat \lambda\in Y_n$ and $\mu,\hat \mu\in Y_{n-1}$ be two
 pairs of Young diagrams, such that both $\lambda$,$\hat \lambda$  and $\mu$,$\hat \mu$ differ
 by the move of box $(i,j)$ into the position $(\hat i,\hat j)$ with $\hat i>i$.
 Further, assume that $\lambda\setminus\mu=\hat \lambda\setminus\hat \mu=(r,c)$, cf.\ Figure
 \ref{Figure_4diag}.
 Fix any integer $N\ge \ell(\hat \lambda)$. If $r<i$ then
 \begin{equation}
 \label{eq_monoton_elementary_1}
  s_\lambda(1^N) s_{\hat \mu}(1^N)\ge s_{\hat \lambda}(1^N) s_{\mu}(1^N).
 \end{equation}
 If $r>\hat i$, then
 \begin{equation}
 \label{eq_monoton_elementary_2}
  s_\lambda(1^N) s_{\hat \mu}(1^N)\le s_{\hat \lambda}(1^N) s_{\mu}(1^N).
 \end{equation}
\end{proposition}
\begin{proof}
 We recall the
 \emph{Weyl dimension formula} (see e.g.\
 \cite[Section 3, Exerceise 1]{M})
 $$
  s_\lambda(1^N)=\prod_{1\le a<b \le N} \frac{\lambda_a-a -
  \lambda_b+b}{b-a}
 $$
 and plug it into \eqref{eq_monoton_elementary_1}. Since $\lambda_a=\hat \lambda_a$ and
 $\mu_a=\hat \mu_a$ for $a\ne i,\hat i$, many factors on the left and right side cancel
 out, and \eqref{eq_monoton_elementary_1} turns into
 \begin{multline*}
  \prod_{\begin{smallmatrix} 1\le a\le N: \\ a \ne i
  \end{smallmatrix}} |\lambda_a-a-\lambda_i+i| \prod_{ \begin{smallmatrix}1\le a\le N: \\ a \ne \hat i\end{smallmatrix}}
  |\lambda_a -a-\lambda_{\hat i}+\hat i|
  \prod_{\begin{smallmatrix} 1\le a\le N: \\ a \ne i
  \end{smallmatrix}}
    |\hat \mu_a-a-\hat \mu_i+i|
  \prod_{\begin{smallmatrix} 1\le a\le N: \\ a \ne \hat i \end{smallmatrix}}
    |\hat \mu_a -a-\hat \mu_{\hat i}+\hat i|
\\ \stackrel{?}{\ge}
  \prod_{\begin{smallmatrix} 1\le a\le N: \\ a \ne i \end{smallmatrix}}
    |\hat \lambda_a-a-\hat \lambda_i+i|
  \prod_{\begin{smallmatrix} 1\le a\le N: \\ a \ne \hat i \end{smallmatrix}}
    |\hat \lambda_a -a-\hat \lambda_{\hat i}+\hat i|
  \prod_{\begin{smallmatrix} 1\le a\le N: \\ a \ne i \end{smallmatrix}}
    |\mu_a-a-\mu_i+i|
 \prod_{\begin{smallmatrix} 1\le a\le N: \\ a \ne \hat i  \end{smallmatrix}}
    |\mu_a -a-\mu_{\hat i}+\hat i|
 \end{multline*}
Since $\lambda_a=\mu_a$ and $\hat \lambda_a=\hat \mu_a$ for $a\ne r$, we can further cancel out the
factors to get
 \begin{multline*}
   (\lambda_r-r-\lambda_i+i)(\lambda_r -r-\lambda_{\hat i}+\hat i)
   (\hat \mu_r-r-\hat \mu_i+i) (\hat \mu_r -r-\hat \mu_{\hat i}+\hat i)
\\ \stackrel{?}{\ge}
    (\hat \lambda_r-r-\hat \lambda_i+i)(\hat \lambda_r -r-\hat \lambda_{\hat i}+\hat i)
    (\mu_r-r-\mu_i+i)(\mu_r -r-\mu_{\hat i}+\hat i).
 \end{multline*}
Rewriting everything in terms of the parts of $\lambda$, we get an
equivalent inequality
 \begin{multline*}
   (\lambda_r-r-\lambda_i+i)(\lambda_r -r-\lambda_{\hat i}+\hat i)
   (\lambda_r-r-\lambda_i+i) (\lambda_r-r-\lambda_{\hat i}+\hat i-2)
\\ \stackrel{?}{\ge}
    (\lambda_r-r-\lambda_i+i+1)(\lambda_r -r-\lambda_{\hat i}+\hat i-1)
    (\lambda_r-r-\lambda_i+i-1)(\lambda_r -r-\lambda_{\hat i}+\hat i-1).
 \end{multline*}
Further transforming, and denoting $\lambda_r-r-\lambda_i+i=x$, $\lambda_r -r-\lambda_{\hat i}+\hat
i-1=y$,
 we
get
 \begin{equation}
 \label{eq_monoton_equivalent}
   x^2( y^2-1) \stackrel{?}{\ge}
    (x^2-1)y^2.
 \end{equation}
Now when $r<i<\hat i$, then $y\ge x>0$ and \eqref{eq_monoton_equivalent} holds. Similarly, when
$r>\hat i>i$, then $0>y\ge x$ and the inequality opposite to $\eqref{eq_monoton_equivalent}$ holds.
\end{proof}
Based on computer computations we believe that the following two generalizations of Proposition
\ref{Prop_elementary_move} should hold.

Recall that a symmetric function $f(x_1,x_2,\dots)$ is called \emph{monomial positive} if the
coefficients of its expansion into monomials are non-negative.
\begin{conjecture} \label{Conj_monomial_positive}
 Let $\lambda,\hat \lambda\in Y_n$ and $\mu,\hat \mu\in Y_{n-1}$ be two
 pairs of Young diagrams, such that both $\lambda$,$\hat \lambda$  and $\mu$,$\hat \mu$ differ
 by the move of box $(i,j)$ into the position $(\hat i,\hat j)$ with $\hat i> i$.
 Further, assume that $\lambda\setminus\mu=\hat \lambda\setminus\hat \mu=(r,c)$, cf.\ Figure
 \ref{Figure_4diag}.
  If $r<i$ then
 $s_\lambda s_{\hat \mu}-s_{\hat \lambda} s_{\mu}$ is monomial--positive.
 If $r>\hat i$, then $s_{\hat \lambda} s_{\mu}-s_{\lambda} s_{\hat \mu}$
 is monomial--positive.
\end{conjecture}
\begin{remark}
 Monomial positivity (and even stronger Schur--positivity) of
 similar quadratic expressions has been intensively studied, see \cite{LPP}, \cite{LP} and
 references therein. However it seems that the differences of the form
 $s_\lambda s_{\hat \mu}-s_{\hat \lambda}s_{\mu}$ are out of the scope of those articles.
\end{remark}

Further, we recall the definition of ($N$--variable version of) Hall--Litlewood symmetric function
on a parameter $t\in\mathbb R$, and a Young diagram $\lambda$ such that $\ell(\lambda)\le N$, cf.
\cite[Chapter III]{M}
$$
 Q_\lambda(x_1,\dots,x_N;t)= (1-t)^{N} \prod\limits_{i=1}^{N-\ell(\lambda)} \frac{1}{1-t^i} \cdot  \sum_{\sigma\in \S(n)} x_{\sigma(1)}^{\lambda_1}\cdots
 x_{\sigma(N)}^{\lambda_N} \prod_{1\le i<j\le N} \frac{x_{\sigma(i)}-t
 x_{\sigma(j)}}{x_{\sigma(i)}-x_{\sigma(j)}}.
$$
Note the normalization that we use, and which is the same as in \cite{M}.

\begin{conjecture} \label{Conjecture_elementary_move}
 Suppose that $0\le t\le 1$ and let $\lambda,\hat \lambda\in Y_n$ and $\mu,\hat \mu\in Y_{n-1}$ be two
 pairs of Young diagrams, such that both $\lambda$,$\hat \lambda$  and $\mu$,$\hat \mu$ differ
 by the move of box $(i,j)$ into the position $(\hat i,\hat j)$ with $\hat i> i$.
 Further, assume that $\lambda\setminus\mu=\hat \lambda\setminus\hat \mu=(r,c)$, cf.\ Figure
 \ref{Figure_4diag}.
 Fix any integer $N\ge \ell(\hat \lambda)$. If $r<i$ then
 \begin{equation}
 \label{eq_HL_monoton_elementary_1}
\left(1-t^{\hat \lambda'_c-\hat \lambda'_{c+1}}\right)   \frac{Q_{\hat \mu}(1^N;t)}{Q_{\hat \lambda}(1^N;t)}
\ge \left(1-t^{\lambda'_c-\lambda'_{c+1}}\right) \frac{Q_{\mu}(1^N;t)}{Q_\lambda(1^N;t)} .
 \end{equation}
 If $r>\hat i$, then
 \begin{equation}
 \label{eq_HL_monoton_elementary_2}
\left(1-t^{\hat \lambda'_c-\hat \lambda'_{c+1}}\right)   \frac{Q_{\hat \mu}(1^N;t)}{Q_{\hat \lambda}(1^N;t)}
\le \left(1-t^{\lambda'_c-\lambda'_{c+1}}\right) \frac{Q_{\mu}(1^N;t)}{Q_\lambda(1^N;t)}.
 \end{equation}
\end{conjecture}
\begin{remark} When $t=0$, Conjecture
\ref{Conjecture_elementary_move} turns into Proposition \ref{Prop_elementary_move}. When $t=1$, the
Hall--Littlewood functions $Q_\lambda(\cdot;t)$ turn into the monomial symmetric functions and the
validity of Conjecture \ref{Conjecture_elementary_move} can be similarly established (in fact,
inequalities turn into equalities in this case). For general $t$ we are not aware of any simple
analogues of the Weyl dimension formula for $P_\lambda(1^N;t)$ and the strategy employed in the
proof of Proposition \ref{Prop_elementary_move} fails.
\end{remark}

\subsection{Proof of Theorem \ref{Theorem_monotonicity}}
The following statement is an immediate corollary of Proposition
\ref{Prop_elementary_move}.
\begin{corollary}
 \label{Cor_elementary_dimensions}
 Let $\lambda,\hat \lambda\in Y_n$ and $\mu,\hat \mu\in Y_{n-1}$ be two
 pairs of Young diagrams, such that both $\lambda$,$\hat \lambda$  and $\mu$,$\hat \mu$ differ
 by the move of box $(i,j)$ into the position $(\hat i,\hat j)$ with $\hat i> i$.
 Further, assume that $\lambda\setminus\mu=\hat \lambda\setminus\hat \mu=(r,c)$, cf.\ Figure
 \ref{Figure_4diag}.
 If $r<i$, then
 \begin{equation}
 \label{eq_monoton_elementary_1_dim}
  \frac{\dim(\hat \mu)}{\dim(\hat \lambda)} \ge \frac{\dim(\mu)}{\dim(\lambda)}.
 \end{equation}
 If $r>\hat i$, then
 \begin{equation}
 \label{eq_monoton_elementary_2_dim}
  \frac{\dim(\hat \mu)}{\dim(\hat \lambda)} \le \frac{\dim(\mu)}{\dim(\lambda)}.
 \end{equation}
\end{corollary}
\begin{proof} The statement follows from Proposition
\ref{Prop_elementary_move} and the limit relation
$$
 \dim(\lambda)=\lim_{N\to\infty}
 \frac{s_\lambda(1^N)}{N^{|\lambda|}}.
$$
The simplest way to prove the latter limit identity is through the explicit formulas for
$\dim(\lambda)$ and $s_\lambda(1^N)$, see e.g.\ \cite[Chapter I, Section 3, Examples 4-5 and
Section 5, Example 2]{M}.

Alternatively, one can  directly prove \eqref{eq_monoton_elementary_1_dim},
\eqref{eq_monoton_elementary_2_dim} along the lines of the proof of Proposition
\ref{Prop_elementary_move}.
\end{proof}

\begin{remark}
 Conjecture \ref{Conj_elementary_dimensions_t} can be obtained from Conjecture \ref{Conjecture_elementary_move} in the 
same way as Corollary \ref{Cor_elementary_dimensions} follows from Proposition \ref{Prop_elementary_move}.
\end{remark}

\begin{definition} \label{Def_cover} For two Young diagrams $\lambda,\hat \lambda\in\Y_n$
we say that $\lambda$ covers $\hat \lambda$ and write $\lambda\succ\hat \lambda$ if $\lambda$ and
$\hat \lambda$ differ by the move of the box $(i,j)\subset\lambda$ into the position $(\hat i,\hat
j)\subset\hat \lambda$ such that either $\hat i-i=1$, or $\hat j-j=-1$.
\end{definition}
An example illustrating Definition \ref{Def_cover} is shown in Figure \ref{Figure_cover}. It is
straightforward to check that if $\lambda\succ\hat \lambda$, then $\lambda\ge\hat \lambda$ and
further $\lambda$ and $\hat \lambda$ are immediate neighbours in the dominance order.

\begin{proof}[Proof of Theorem \ref{Theorem_monotonicity}]
It suffices to consider the case $k=n-1$, as the case of general $k<n$ would follow from the former
by induction. Further, due to the definition of the relation $\rho\ge \hat \rho$, it suffices to
consider the case when both these measures are atoms, i.e.\ $sup(\rho)=\lambda$ and $sup(\hat
\rho)=\hat \lambda$ with $\lambda\ge\hat \lambda$. Further, since the dominance order and
stochastic dominance relation are transitive, it suffices to consider the case when $\lambda$ and
$\hat \lambda$ are immediate neighbors in the partial order, i.e.\ $\lambda\succ\hat \lambda$.
Without loss of generality we assume that $\lambda$ and $\hat \lambda$ differ by the move of the
box $(i,j)\subset\lambda$ into the position $(\hat i,\hat j)\subset\hat \lambda$ such that $\hat
i-i=1$.

In the latter case $\pi^n_{n-1}(\rho)$ assigns the mass
\begin{equation}
\label{eq_x1}
 \frac{\dim(\mu)}{\dim(\lambda)}
\end{equation}
to each diagram $\mu\in\Y_{n-1}$, such that $\mu\pr\lambda$. Similarly, $\pi^n_{n-1}(\hat \rho)$
assigns the mass
\begin{equation}
\label{eq_x2}
 \frac{\dim(\hat \mu)}{\dim(\hat \lambda)}
\end{equation}
to each diagram $\hat \mu\in\Y_{n-1}$, such that $\hat \mu\pr\hat \lambda$. Subdivide all
$\mu\in\Y_{n-1}$, such that $\mu\pr\lambda$ into three sets
$$
 A^{\uparrow}_\lambda=\{\mu\in\Y_{n-1}\mid \lambda\setminus\mu=(r,c),\,
 r<
 i\},\quad A^{\downarrow}_\lambda=\{\mu\in\Y_{n-1}\mid \lambda\setminus\mu=(r,c),\,
 r>
 \hat i\},
$$
$$
 A^{=}_\lambda=\{\mu\in\Y_{n-1}\mid \lambda\setminus\mu=(r,c),\, i\le r\le \hat i\}.
$$
Now for $\mu\in A^{\uparrow}_\lambda\cup A^\downarrow_\lambda$ set $\hat \mu\prec\mu$ to be the
Young diagram obtained by moving the box $(i,j)$ into the position $(\hat i,\hat j)$. Now the
following three observations imply the stochastic dominance $\pi^n_{n-1}\rho\ge\pi^n_{n-1}\hat
\rho$:
\begin{itemize}
 \item All the Young diagrams from $A^{\uparrow}_\lambda\cup  A^{=}_\lambda
 \cup A^\downarrow_\lambda$ are linearly ordered (with respect to
 the dominance order) by $r$, which is the row number of the
 box being removed from $\lambda$. The same is true for $A^{\uparrow}_{\hat \lambda}\cup  A^=_{\hat
 \lambda}
 \cup A^\downarrow_{\hat \lambda}$.
 \item Each Young diagram from $A^{\downarrow}_\lambda\cup
A^=_\lambda$ dominates each Young diagram from $A^{\uparrow}_{\hat \lambda}\cup A^=_{\hat
\lambda}$.
 \item Due to Corollary \ref{Cor_elementary_dimensions} and formulas
 \eqref{eq_x1}, \eqref{eq_x2}, for each $\mu\in
 A^{\uparrow}_\lambda$ we have $(\pi^n_{n-1}\rho)(\mu)\le
 (\pi^n_{n-1}\hat \rho)(\hat \mu)$ and for each $\mu\in
 A^{\downarrow}_\lambda$ we have $(\pi^n_{n-1}\rho)(\mu)\ge
 (\pi^n_{n-1}\hat \rho)(\hat \mu)$. \qedhere
\end{itemize}
\end{proof}

\section{The Law of Large Numbers for the Young graph}
\label{Section_LLN}

The second ingredient of our proof of the Thoma theorem (Theorem \ref{Theorem_Thoma}) is the Law of
Large Numbers for the measures appearing in its formulation.
\begin{theorem}[The law of large numbers, \cite{VK_S}, \cite{KOO}, \cite{Bufetov}, \cite{Meliot}]
 \label{Theorem_LLN} Choose two strictly decreasing finite sequences $\alpha_1>\alpha_2>\dots>\alpha_a>0$,
 $\beta_1>\beta_2>\dots>\beta_b>0$ such that $\sum_{i=1}^a \alpha_i +\sum_{i=1}^b \beta_i =
 1$.

 For $n=1,2,\dots$ let $\lambda(n)\in\Y_n$ be a random Young diagram distributed according to
 the probability measure
 $$
  M_n^{(\alpha,\beta)}(\lambda)=\dim(\lambda) s_\lambda(\alpha,\beta).
 $$
 Then for each $i=1,\dots,a$ and each $j=1,\dots,b$ we have (in probability)
 $$
  \lim_{n\to\infty} \frac{\lambda_i(n)}{n}=\alpha_i,\quad \lim_{n\to\infty}
  \frac{\lambda'_j(n)}{n}=\beta_j.
 $$
\end{theorem}
\begin{remark}
 In fact, an analogue of Theorem \ref{Theorem_LLN} holds for all extreme measures of Theorem
 \ref{Theorem_Thoma}, see \cite{VK_S}, \cite{KOO}, \cite{Bufetov}, \cite{Meliot}. However, the
 present weaker form is enough for our purposes.
\end{remark}

There are at least four different approaches in the literature to the proof of Theorem
\ref{Theorem_LLN}:
\begin{itemize}
 \item The proof of the Thoma theorem in \cite{VK_S}, \cite{KOO} based on the relation of the
 dimensions in Young graph to the \emph{shifted} Schur functions, as a byproduct
 implies Theorem \ref{Theorem_LLN}. Note that we would like to avoid using
 this approach here, since our aim is to produce an independent proof of Thoma theorem.

 \item Vershik and Kerov in \cite{VK_RSK} showed how the random Young diagrams $\lambda(n)$ can be
  sampled using (a modification of) the classical Robsinson--Schensted correspondence, whose input
  is a sequence of $n$ i.i.d.\ discrete random variables. This observation allows to deduce Theorem
  \ref{Theorem_LLN} from the conventional Law of Large Numbers for sequences of independent
  random variables. For the details we refer to \cite{Bufetov}, where, in fact, a stronger Cental
  Limit Theorem was proved using this approach.

 \item Kerov explained in \cite{Kerov_CLT} (see also \cite{IvanovOlshanski}) how certain observables of random Young diagrams
 $\lambda(n)$ can be computed using the algebra of shifted--symmetric functions. The resulting
 formulas turn out to be well-suited for the asymptotics analysis along the lines of
 Theorem \ref{Theorem_LLN}, which was done in \cite{Meliot}.
 In fact, \cite{Meliot} also proves a stronger Central Limit Theorem.

 \item Following the approach of \cite{J-Annals}, \cite{BOO}, \cite{Ok_Schur}
 one proves that the \emph{poissonization} of measures $M_n$ can be described via
 a \emph{determinantal point process}, with an explicit contour integral expression for the kernel.
 Asymptotic analysis of this kernel via steepest descent gives Theorem \ref{Theorem_LLN}.
\end{itemize}

Each of the above four methods for proving Theorem \ref{Theorem_LLN} relies on a certain very
nontrivial (but known)  technique, which is the algebra of shifted--symmetric functions for the
first and third  approaches, the Robinson--Schensted correspondence for the second approach and
determinantal point processes / Schur measures for the forth one. Given the knowledge of this
technique the proof of Theorem \ref{Theorem_LLN} becomes relatively simple.

We now give a sketch of the second ``combinatorial'' proof of Theorem \ref{Theorem_LLN}, which is
based on the Robinson--Schensted correspondence.

\begin{proof}[Sketch of the proof of Theorem \ref{Theorem_LLN}]

Let us consider an alphabet $\mathcal T = \mathcal T^+ \cup \mathcal T^-$, where $\mathcal T^+ = \{
t_1^+, \dots, t_a^+ \}$ and $\mathcal T^- = \{ t_1^-, \dots, t_b^- \}$. Let us fix a linear order
on $\mathcal T$; its exact choice is irrelevant, so e.g.\ one can assume that $$
 t_{b}^{-}<t_{b-1}^{-}<\dots<t_1^{-}<t_1^{+}<\dots<t_q^{+}.
$$
For $x,y \in \mathcal T$ we write $x \lhd y$ if either $x < y$, or  $x=y \in \mathcal T^+$. We
write $x \rhd y$ if either $x>y$ or $x=y\in\mathcal T^{-}$. We call a word $x_1 \dots x_n \in
\mathcal A^n$ \textit{increasing} if $x_1 \lhd x_2 \lhd \dots \lhd x_n$, and \textit{decreasing} if
$x_1 \rhd x_2 \rhd \dots \rhd x_n$. For a word $w$ let us denote by $r_s (w)$ the maximal
cardinality of the union of $s$ disjoint increasing subsequences of the word $w$, and by $c_s (w)$
the maximal cardinality of the union of $s$ disjoint decreasing subsequences.

Now let us define the probability measure $\eta^{(\alpha,\beta)}$ on $\mathcal T$ such that
$\eta^{(\alpha,\beta)}(a_i) = \alpha_i$ and $\eta^{(\alpha,\beta)}(b_j) = \beta_j$. Let $w(n)$,
$n=1,2\dots$ be a random element of $\mathcal T^n$ distributed according to the product measure
$(\eta^{(\alpha,\beta)})^{\otimes n}$. Vershik-Kerov \cite{VK_RSK} relying on a generalization of
Robinson--Schensted correspondence (see also \cite{BereleRegev}) proved that the following equality
in distribution holds jointly for all $s=1,2,\dots$
\begin{equation}
\label{eq_RSK_identity} \lambda_1 (n) + \dots + \lambda_s (n) \stackrel{d}{=} r_s (w(n)), \qquad
\lambda_1'(n) + \dots + \lambda_s' (n) \stackrel{d}{=} c_s (w(n)).
\end{equation}
The identity \eqref{eq_RSK_identity} reduces Theorem \ref{Theorem_LLN} to the Law of Large Numbers
as $n\to\infty$ for $r_s(w(n))$ and $c_s(w(n))$, $s=1,2,\dots$. The latter is rather transparent.
Indeed, it is intuitively clear that the length of the longest increasing subsequence in the word
$w(n)$ should be (up to a small error) equal to the length of the subsequence of all letters
$t_1^+$ in $w(n)$, and the last length is approximately $\alpha_1 \cdot n$ due to the classical Law
of Large Numbers for independent random variables. Further, the main contribution to $r_s(w(n))$
comes when each subsequence contains only one letter from our alphabet, and thus $r_s(w(n))\approx
(\alpha _1+\dots+\alpha_s)\cdot n$ for $1\le s\le a$. Similarly, $c_s(w(n))\approx
(\beta_1+\dots,\beta_s)\cdot n$ for $1\le s\le b$. A formal proof based on this argument is given
in \cite[Theorem 2]{Bufetov}, see also \cite[Section 6]{Meliot} and \cite[Theorem 6.4]{Sniady_RSK}.
\end{proof}

\section{Proof of Thoma theorem}
\label{Section_Thoma}

We start by explaining informally the main idea behind the proof of Theorem \ref{Theorem_Thoma}.

For any $\lambda\in\Y_k$ we define a probability measure $\rho_{\lambda}$ on $\Y_k$ to be an atom
with support $\sup(\rho_{\lambda})=\lambda$. We start the proof from an abstract convex analysis
statement (Proposition \ref{Lemma_ergodic_method}) that any extreme coherent system $\{M_n\}$ can
be approximated by systems of the form $\pi_n^k(\rho_{\lambda(k)})$ for a sequence
$\lambda(k)\in\Y_k$, $k=1,2,\dots$. We further use the Law of Large Numbers to show in Lemma
\ref{Lemma_between} that when $k$ is large enough and after dropping out a tiny mass $\eps$, the
measure $\rho_{\lambda(k)}$ can be clutched  between two measures $M_{k}^{(\alpha^-,\beta^-)}$ and
$M_{k}^{(\alpha^+,\beta^+)}$. Moreover, they can be chosen so that the distance between
$(\alpha^-,\beta^-)$ and $(\alpha^+,\beta^+)$ is small. Now Theorem \ref{Theorem_monotonicity}
implies that $\pi_n^k(\rho_{\lambda(k)})$ is clutched between $M_{n}^{(\alpha^-,\beta^-)}$ and
$M_{n}^{(\alpha^+,\beta^+)}$. At this point we conclude that any coherent system $\{M_n\}$ can be
well-approximated by the coherent systems of the form $\{M_{n}^{(\alpha,\beta)}\}$,
$(\alpha,\beta)\in\Omega$. Therefore, the closedness of the latter set of coherent systems implies
Theorem \ref{Theorem_Thoma}.

The formal proof of Theorem \ref{Theorem_Thoma} is given at the end of this section after we
present
 a series of auxiliary statements.

\smallskip

\begin{proposition} \label{Lemma_ergodic_method}
 Let $\{M_n\}_{n=1}^{\infty}$ be an extreme coherent system of measures. Then there exists a
 (deterministic) sequence of Young diagrams $\lambda(k)\in\Y_k$, $k=1,2,\dots$ such that
 \begin{equation}
 \label{eq_approximation}
  M_n=\lim_{k\to\infty} \pi_n^k(\rho_{\lambda(k)}),\quad n=1,2,\dots.
 \end{equation}
\end{proposition}
\begin{proof}
 This is a particular case of a very general convex analysis statement, which was reproved many times in different
 contexts. Its first appearance in the asymptotic representation theory dates back to
 \cite{V_ergodic}, since then it is known as ``ergodic method''. The complete proofs of the statements generalizing Proposition \ref{Lemma_ergodic_method} can be found in \cite[Section
 6]{OkOlsh} or \cite[Theorem 1.1]{DF}.
\end{proof}

Recall that for two measures $\rho$, $\hat \rho$ on a finite set $A$, their total variation
distance is defined through
$$
 \d(\rho,\hat \rho)=\frac{1}{2}\sum_{a\in A} |\rho(a)-\hat \rho(a)|.
$$
We also define the $L_\infty$ distance between two pairs of sequences
$(\alpha,\beta)=(\alpha_1\ge\alpha_2\ge\dots,\beta_1\ge\beta_2\ge\dots)$, $(\hat \alpha,\hat
\beta)=(\hat \alpha_1\ge\hat \alpha_2\ge\dots,\hat \beta_1\ge\hat \beta_2\ge\dots)$ through
$$
 \di((\alpha,\beta),(\hat \alpha,\hat \beta))=\max\left( \sup_{i}|\alpha_i-\hat \alpha_i|,\, \sup_i
 |\beta_i-\hat \beta_i| \right).
$$

The following two lemmas explain that the metrics $\d$ on probability measures on $\Y_n$ and $\di$
on $\Omega$ are compatible.

\begin{lemma} \label{Lemma_convergence_on_the_level} For any $n=1,2,\dots$ we have
 \begin{equation}
 \label{eq_x3}
  \lim_{\eps\to 0} \sup_{\begin{smallmatrix}(\alpha,\beta),(\hat \alpha,\hat \beta)\in \Omega:\\
  \di((\alpha,\beta),(\hat \alpha,\hat \beta))\le \eps \end{smallmatrix}}
  \d(M_n^{(\alpha,\beta)},M_n^{(\hat \alpha,\hat \beta)})=0.
 \end{equation}
\end{lemma}
\begin{proof} Note that $\Y_n$ is a finite, therefore it suffices to prove \eqref{eq_x3} with $\d$
replaced by $|M_n^{(\alpha,\beta)}(\lambda)-M_n^{(\hat \alpha,\hat \beta)}(\lambda)|$ for arbitrary
$\lambda\in\Y_n$. Moreover, due to the definition \eqref{eq_ext_measure_def}, it suffices to study
$|s_\lambda(\alpha,\beta)-s_\lambda(\hat \alpha,\hat \beta)|$. To analyze this difference recall
that the Schur function $s_\lambda$ is a polynomial in power sums $p_1,\dots,p_n$, which generate
the algebra of symmetric functions. We conclude that \eqref{eq_x3} is equivalent to
 \begin{equation}
 \label{eq_x4}
  \lim_{\eps\to 0} \sup_{\begin{smallmatrix}(\alpha,\beta),(\hat \alpha,\hat \beta)\in \Omega:\\
  \di((\alpha,\beta),(\hat \alpha,\hat \beta))\le \eps \end{smallmatrix}}
  |p_n(\alpha,\beta)-p_n(\hat \alpha,\hat \beta)|=0, \quad n=1,2,\dots.
 \end{equation}
 To prove \eqref{eq_x4} we recall the definition \eqref{eq_specialization} and first conclude that
 $$
  |p_1(\alpha,\beta)-p_1(\hat \alpha,\hat \beta)|=|1-1|=0.
 $$
 Further, for $n>1$ we have
 \begin{multline}
    |p_n(\alpha,\beta)-p_n(\hat \alpha,\hat \beta)|\le \sum_{i=1}^\infty
    |\alpha_i-\hat \alpha_i|\bigl( (\alpha_i)^{n-1}+(\alpha_i)^{n-2}
    (\hat \alpha_i)^1+\dots+(\hat \alpha_i)^{n-1} \bigr)
     \\ + \sum_{i=1}^\infty
    |\beta_i-\hat \beta_i| \bigl( (\beta_i)^{n-1}+(\beta_i)^{n-2} (\hat \beta_i)^1+\dots+(\hat \beta_i)^{n-1}
    \bigr)
   \\ \le \di((\alpha,\beta),(\hat \alpha,\hat \beta))\cdot n \sum_{i=1}^n \left[ (\alpha_i)^{n-1}+
    (\hat \alpha_i)^{n-1}+ (\beta_i)^{n-1}+ (\hat \beta_i)^{n-1} \right]\\ \le 4n \cdot
    \di((\alpha,\beta),(\hat \alpha,\hat \beta)),
 \end{multline}
which immediately implies \eqref{eq_x4}.
\end{proof}

\begin{lemma} \label{Lemma_closed}
 Let $(\alpha(k),\beta(k))$, $k=1,2,\dots$ be pairs of sequences. Suppose that for each
 $n=1,2,\dots$ the measures $M_n^{(\alpha(k),\beta(k))}$ converge in the sense of $\d$ to a
 measure $M_n$. Then there exists a pair of sequences $(\alpha,\beta)$ such that
 $M_n=M_n^{(\alpha,\beta)}$ for all $n$.
\end{lemma}
\begin{proof} We first claim that $\Omega$ is a compact set in the topology defined by $\di$.
Indeed, this topology on $\Omega$ is equivalent to the topology of pointwise convergence. For the
latter topology $\Omega$ is compact, since it is a closed subset of the compact set
$[0,1]^{\infty}$. Now we define $(\alpha,\beta)$ as a limiting point of the sequence of pairs
$(\alpha(k),\beta(k))$, $k=1,2,\dots$. Using Lemma \ref{Lemma_convergence_on_the_level} we conclude
that $M_n=M_n^{(\alpha,\beta)}$ for all $n$. \end{proof}

The next lemma is the key point of our proof of Theorem \ref{Theorem_Thoma}.

\begin{lemma} \label{Lemma_between}
 Take a sequence of integers $0<k(1)<k(2)<\dots$ and
 let $\lambda(n)\in\Y_{k(n)}$, $n=1,2,\dots$ be a sequence of Young diagrams such that the following
 limits exist for each $i=1,2,\dots$
 $$
  \lim_{n\to\infty} \frac{\lambda_i(n)}{k(n)}=\alpha_i,\quad \lim_{n\to\infty}
  \frac{\lambda'_i(n)}{k(n)}=\beta_i.
 $$
 Then for every $\eps>0$ and every $N\in\mathbb N$ there exists $n>N$, two measures $\rho^+_n$,
 $\rho^-_n$ on $\Y_{k(n)}$ and two pairs of sequences $(\alpha^+,\beta^+),(\alpha^-,\beta^-)\in\Omega$, such
 that
 \begin{enumerate}
  \item $\d\left(\rho^-_n,M_{k(n)}^{(\alpha^-,\beta^-)}\right)<\eps$ and $\d\left(\rho^+_n,M_{k(n)}^{(\alpha^+,\beta^+)}\right)<\eps$,
  \item $\di((\alpha^-,\beta^-),(\alpha^+,\beta^+))<\eps$,
  \item $\rho^-_n\le \rho_{\lambda(n)}\le \rho^+_n$ in the sense of stochastic dominance.
 \end{enumerate}
\end{lemma}
In words, Lemma \ref{Lemma_between} says that the delta--measure on a Young diagram of a large
level $\Y_k$ (after dropping a tiny mass $\eps$) can be always clutched between two measures
$M_{k}^{(\alpha^-,\beta^-)}$ and $M_{k}^{(\alpha^+,\beta^+)}$. Moreover, they can be chosen so that
the distance between $(\alpha^-,\beta^-)$ and $(\alpha^+,\beta^+)$ is small. The proof relies on
the Law of Large Numbers for the measures $M_k^{(\alpha,\beta)}$.
\begin{proof}[Proof of Lemma \ref{Lemma_between}] Take $L_\alpha,L_\beta > 0$ such that $\alpha_{L_\alpha}<\eps/2$ and
$\beta_{L_\beta}<\eps/2$, but $\alpha_i\ge \eps/2$ for all $i<L_\alpha$ and $\beta_j\ge\eps/2$ for
all $j<L_\beta$. Further choose $V>2$, such that $\alpha_{L_\alpha}<\eps/2-\eps/V$ and
$\beta_{L_\beta}<\eps/2-\eps/V$. We will now define the pair of sequences $(\alpha^+,\beta^+)$ as
follows.
$$
 \alpha^+_i=\alpha_i+\frac{\eps}{V\cdot 2^{i}}, \quad i=2,\dots,L_\alpha,\quad\quad
 \beta^+_j=\beta_j-\frac{\eps}{V\cdot 2^{L_\beta+1-j}}, \quad j=1,\dots,L_\beta.
$$
For $j>L_\beta$ we set $\beta^+_i=0$. For $i=L_{\alpha}+1,\dots,R$ we set
$\alpha_i^+=\eps/2-\eps/V+\frac{\eps}{V 2^i}$ where $R$ is the minimum integer such that
$$
S(R):= \left(\alpha_1+\frac{\eps}{2 V}\right) + \sum_{i=2}^{R+1} \alpha^+_i +\sum_{j=1}^{L_\beta}
\beta^+_j >1.
$$
Finally, set $\alpha_1^+=\alpha_1+\frac{\eps}{2 V}+(1-S(R-1))$ and $\alpha_i=0$ for $i>R$.

Note that the resulting $(\alpha^+,\beta^+)$ satisfies the assumptions of Theorem
\ref{Theorem_LLN}. Combining this theorem with the definition of numbers $\alpha_i$, $\beta_i$, we
conclude the existence of $N_1$ such that for all $n>N_1$ the diagram $\lambda(n)\in \Y_{k(n)}$ is
dominated by $M^{(\alpha^+,\beta^+)}_{k(n)}$--random Young diagram $\mu(n)$ with probability
greater than $(1-\eps)$. Thus, if we define $\rho^+(n)$ on $\Y_{k(n)}$ through the identity
$$
 \rho^+_n(\mu)=\begin{cases} M^{(\alpha^+,\beta^+)}_{k(n)}(\mu),& \mu> \lambda(n),\\
                             1-\sum_{\nu>\lambda(n)} M^{(\alpha^+,\beta^+)}_{k(n)}(\nu), &
                             \mu=\lambda(n),\\
                             0,& \text{otherwise.},
               \end{cases}
$$
then both $\d\left(\rho^+,M_{k(n)}^{(\alpha^+,\beta^+)}\right)<\eps$ and  $\rho_{\lambda(n)}\le
\rho^+_n$ hold.

Arguing similarly but with the roles of $\alpha$'s and $\beta$'s switched, we define
$(\alpha^-,\beta^-)$ and $\rho^-(n)$. It remains to note that
$$
 \di((\alpha^-,\beta^-),(\alpha^+,\beta^+))\le \di((\alpha^-,\beta^-),(\alpha,\beta))+
 \di((\alpha,\beta),(\alpha^+,\beta^+)) < \eps/2+\eps/2=\eps. \qedhere
$$
\end{proof}

\begin{proof}[Proof of Theorem \ref{Theorem_Thoma}] Let $\{M_r\}_{r=1}^{\infty}$ be an extreme coherent system of measures and let $\lambda(k)\in\Y_k$, $k=1,2,\dots$
 be a corresponding sequence of Young diagrams as in Proposition \ref{Lemma_ergodic_method}. Since for all $i=1,2,\dots$, we have
 $0\le \lambda_i(k)/k \le 1$ and $0\le \lambda_i'(k) \le 1 $, passing to a
 subseqence $k(n)$, $n=1,2,\dots$ we can assume that the following limits exist
 $$
  \lim_{n\to\infty} \frac{\lambda_i(k(n))}{k(n)}=\alpha_i,\quad \lim_{n\to\infty}
  \frac{\lambda'_i(k(n))}{k(n)}=\beta_i.
 $$
 Now we choose $\eps(n)=1/n$. Passing, if necessary, to another subsequence (which we will denote by the same $k(n)$ to avoid complicating the notations) and using Lemma
 \ref{Lemma_between}, we conclude that there exist $(\alpha^-(n),\beta^-(n)),
 (\alpha^+(n),\beta^+(n))\in\Omega$ and measures $\rho^+(n),\rho^-(n)$ on $\Y_{k(n)}$ such that
 \begin{enumerate}
  \item $\d\left(\rho^-(n),M_{k(n)}^{(\alpha^-(n),\beta^-(n))}\right)<\frac1n$ and $\d\left(\rho^+(n),M_{k(n)}^{(\alpha^+(n),\beta^+(n))}\right)<\frac1n$,
  \item $\di\bigl((\alpha^-(n),\beta^-(n)),\,(\alpha^+(n),\beta^+(n))\bigr)<\frac1n$,
  \item $\rho^-(n)\le \rho_{\lambda(k(n))}\le \rho^+(n)$ in the sense of stochastic dominance.
 \end{enumerate}
Now choose any $r=1,2,\dots$. We aim to prove that $M_r=\lim_{n\to\infty}
M_{r}^{(\alpha^-(n),\beta^-(n))}$ in the sense of $\d$. For that note that since each map $\pi^m_k$
is a contraction in $\d$ distance, Lemma \ref{Lemma_convergence_on_the_level} implies as
$n\to\infty$
\begin{multline}
\label{eq_x5}
 \d(\pi^{k(n)}_r\rho^-(n),\pi^{k(n)}_r\rho^+(n))\le
 \d(\pi^{k(n)}_r\rho^-(n),M_r^{(\alpha^-(n),\beta^-(n))})\\+
 \d\left(M_r^{(\alpha^-(n),\beta^-(n))},M_r^{(\alpha^+(n),\beta^+(n))}\right)+\d(M_r^{(\alpha^+(n),\beta^+(n))},\pi^{k(n)}_r\rho^+(n))
\\ \le
 \frac{2}{n}+
 \d\left(M_r^{(\alpha^-(n),\beta^-(n))},M_r^{(\alpha^+(n),\beta^+(n))} \right) \to 0.
\end{multline}
We claim that the last inequality implies that
\begin{equation}
\label{eq_x6}  \d\left(\pi^{k(n)}_r\rho^-(n), \pi^{k(n)}_r\rho_{\lambda(k(n))} \right) \to 0.
\end{equation}
Indeed, by Theorem \ref{Theorem_monotonicity}
$$
 \pi^{k(n)}_r\rho^-(n)\le \pi^{k(n)}_r \rho_{\lambda(k(n))} \le \pi^{k(n)}_r\rho^+(n).
$$
Thus, for any upper\footnote{By the definition an upper set $U$ in a partially ordered set $A$
satisfies the property that if $x\in U$ and for some $y\in A$ we have $x<y$, then also $y\in U$.}
set $U\subset \Y_r$ we have
$$
 \pi^{k(n)}_r\rho^-(n)(U)\le \pi^{k(n)}_r \rho_{\lambda(k(n))} (U) \le \pi^{k(n)}_r\rho^+(n)(U).
$$
Therefore, as $n\to\infty$
\begin{multline} \label{eq_x7}
 |\pi^{k(n)}_r\rho^-(n)(U) - \pi^{k(n)}_r \rho_{\lambda(k(n))}(U)|\le | \pi^{k(n)}_r\rho^-(n)(U) -\pi^{k(n)}_r\rho^+(n)(U)
 |
 \\ \le \d(\pi^{k(n)}_r\rho^-(n),\pi^{k(n)}_r\rho^+(n)) \to 0.
\end{multline}
Note that for any $\lambda \in \Y_r$ both $\{ \mu\in\Y_r: \mu \ge \lambda \}$ and $\{ \mu\in\Y_r:
\mu
> \lambda \}$ are upper sets, whose difference is $\{\lambda\}$. Therefore, \eqref{eq_x7} implies \eqref{eq_x6}. Now combining
\eqref{eq_x6} with \eqref{eq_approximation} and with inequality
$\d(\pi^{k(n)}_r\rho^-(n),M_r^{(\alpha^-(n),\beta^-(n))})\le 1/n$, we prove that
$$
 M_r=\lim_{n\to\infty}
M_{r}^{(\alpha^-(n),\beta^-(n))}.
$$
Now it remains to apply Lemma \ref{Lemma_closed}.
\end{proof}

\end{document}